\documentclass[11pt, reqno]{amsart}
\usepackage{graphicx} 
\usepackage{amsmath}
\usepackage{amsthm}
\usepackage{amssymb}
\usepackage{color}
\usepackage{tikz-cd}
\usepackage{verbatim}
\usepackage{stmaryrd}
\usepackage{hyperref}
\usepackage{savesym}
\savesymbol{degree}
\usepackage{mathabx}
\restoresymbol{idk}{degree}
\usepackage{scalerel}
\usepackage{mathrsfs}
\usepackage{biblatex} 
\usepackage{pifont}

\usepackage{amsthm}

\usepackage{geometry}
\newtheorem{theorem}{Theorem}[section]
\newtheorem{definition}[theorem]{Definition}
\newtheorem{proposition}[theorem]{Proposition}
\newtheorem{lemma}[theorem]{Lemma}
\newtheorem{corollary}[theorem]{Corollary}
\newtheorem{fact}[theorem]{Fact}

\theoremstyle{remark}

\title{Points and their multiples on curves in powers of simple abelian varieties}
\author{David J. Smith}

\begin{document}

\begin{abstract}
    Let $G$ be a simple abelian variety of dimension $g \in \mathbb{N}$ defined over $\mathbb{Q}^\mathrm{alg}$ and let $C_1, C_2 \subseteq G^N(\mathbb{C})$ be irreducible closed algebraic curves with $N \geq 3$. Further assume that at least one of $C_1$ and $C_2$ is not defined over $\mathbb{Q}^\mathrm{alg}$. Suppose that there does not exist an algebraic subgroup $H_1 \subseteq G^N(\mathbb{C})$ of dimension $g$ such that $C_1 \subseteq H_1$ and that there does not exist an algebraic subgroup $H_2 \subseteq G^N(\mathbb{C})$ of dimension $2g$ such that $C_1 \cup C_2 \subseteq H_2$. Denoting $\mathcal{N} = \{n \in \mathbb{N} \ | \ [n]C_1 \subseteq C_2\}$, we prove that $\bigcup_{n \in \mathbb{N} \setminus \mathcal{N}}\{x \in C_1 \ | \ nx \in C_2\}$ is finite. 
\end{abstract}

\maketitle

\section{Introduction}

Suppose we are given two algebraic curves $C_1, C_2$ in $\mathbb{G}_m^n(\mathbb{C})$. What can we say about the points $x \in C_1$ such that $x^m$ lies in $C_2$? This question was first posed by A. Levin (see \cite[p.39]{zannierBook}) and, by recasting this problem as an intersection of $C_1 \times C_2$ with a family of algebraic subgroups, fits naturally within the paradigm of unlikely intersections. In particular, Zannier observed in \cite[p.~39-40]{zannierBook}, assuming the Zilber-Pink conjecture for algebraic tori, 
that there are finitely many such points outside certain exceptional cases forced by algebraic subgroup relations.

An early partial result was given by Bays and Habegger in \cite{BaysHabegger2015} where they considered the case $C_1 = C_2$ for curves defined over a number field and proved a finiteness result under the additional hypothesis that some Galois conjugate of $C_1$ has finite intersection with the maximal compact subgroup of $\mathbb{G}_m^N(\mathbb{C})$. More recently, Boxall in \cite{boxall2024points} proved the following. 

\begin{fact}\label{OGTheorem}
    Let $C_1, C_2 \subseteq \mathbb{G}^N_m(\mathbb{C})$ be irreducible closed algebraic curves with $N \geq 3$. Further assume that at least one of $C_1$ and $C_2$ is not defined over $\mathbb{Q}^\mathrm{alg}$. Suppose that there does not exist an algebraic subgroup $H_1 \subseteq \mathbb{G}^N_m(\mathbb{C})$ of dimension one such that $C_1 \subseteq H_1$ and that there does not exist an algebraic subgroup $H_2 \subseteq \mathbb{G}^N_m(\mathbb{C})$ of dimension two such that $C_1 \cup C_2 \subseteq H_2$. Let $\mathcal{N} = \{n \in \mathbb{N} \ | \ [n]C_1 \subseteq C_2\}$. Then $\bigcup_{n \in \mathbb{N} \setminus \mathcal{N}}\{x \in C_1 \ | \ x^n \in C_2\}$ is finite. 
\end{fact}

Boxall's strategy was to first reduce the above fact using well-known consequences of Mordell-Lang to the following.
\begin{fact}\label{OGGmTheoremReduced}
    Let $C_1, C_2 \subseteq \mathbb{G}_m^3(\mathbb{C})$ be irreducible closed algebraic curves such that at least one of the curves is not defined over $\mathbb{Q}^\mathrm{alg}$. Suppose that there does not exist a proper algebraic subgroup $H \subseteq \mathbb{G}_m^3(\mathbb{C})$ such that $C_1 \subseteq H$. Let $\mathcal{N} = \{n \in \mathbb{N} \ | \ [n]C_1 \subseteq C_2\}$. Then $\bigcup_{n \in \mathbb{N}\setminus\mathcal{N}}\{x \in C_1 \ | \ x^n \in C_2\}$ is finite.
\end{fact} 
To prove this reduced case, Boxall worked within the machinery of partial exponential fields, developed by Kirby in \cite{KirbyAlgebraic}. He slightly adapted the proof of a transcendence bound of Bays, Kirby and Wilkie, an analogue of Schanuel's conjecture for raising to an exponentially transcendental power, for a marginally more general result and used this bound along with some model-theoretic techniques and o-minimality to obtain the above fact.

In this paper, we present a natural analogue to Fact \ref{OGTheorem}  in the case of powers of simple abelian varieties and prove the following. \begin{theorem}\label{OGEllipticConjecture}
     Let $G$ be a simple abelian variety of dimension $g \in \mathbb{N}$ defined over $\mathbb{Q}^\mathrm{alg}$ and let $C_1, C_2 \subseteq G^N(\mathbb{C})$ be irreducible closed algebraic curves with $N \geq 3$. Further assume that at least one of $C_1$ and $C_2$ is not defined over $\mathbb{Q}^\mathrm{alg}$. Suppose that there does not exist an algebraic subgroup $H_1 \subseteq G^N(\mathbb{C})$ of dimension $g$ such that $C_1 \subseteq H_1$ and that there does not exist an algebraic subgroup $H_2 \subseteq G^N(\mathbb{C})$ of dimension $2g$ such that $C_1 \cup C_2 \subseteq H_2$. Let $\mathcal{N} = \{n \in \mathbb{N} \ | \ [n]C_1 \subseteq C_2\}$. Then $\bigcup_{n \in \mathbb{N} \setminus \mathcal{N}}\{x \in C_1 \ | \ nx \in C_2\}$ is finite. 
\end{theorem}
Both Fact \ref{OGTheorem} and Theorem \ref{OGEllipticConjecture} may be viewed as generic forms of Zilber-Pink due to their additional transcendence assumptions. Results of this kind  have appeared in a number of related settings. Pila and Scanlon proved function field versions of Zilber-Pink for variations of Hodge structures using differential-algebraic methods in \cite{pila2021effectivetranscendentalzilberpinkvariations}. In the abelian variety setting, Barroero and Dill in \cite{barroero2019zilberpink} reduced Zilber-Pink over an algebraically closed field of characteristic 0 to the corresponding problem over $\mathbb{Q}^\mathrm{alg}$ using the $\mathbb{Q}^\mathrm{alg}$-trace. In particular, for the corresponding statement of Theorem \ref{OGEllipticConjecture} where $G$ is not defined over $\mathbb{Q}^\mathrm{alg}$, Barroero and Dill cover this case in \cite[Corollary 1.7.]{barroero2019zilberpink}. More generally in \cite{BarroeroDill2025}, they also extended these results to their framework of distinguished categories which includes commutative algebraic groups and mixed Shimura varieties.

More recently, Klingler and Tayou in \cite{klingler2025zilberpink} proved Zilber-Pink for subvarieties of mixed Shimura varieties satisfying a non-definability condition over $\mathbb{Q}^\mathrm{alg}$ on all non-trivial quotient images and, using similar techniques to that of Pila and Scanlon, Aslanyan, Eterovi\'c and Fowler in \cite{aslanyan2025modular} established a modular Zilber-Pink result for varieties under a related non-definablity condition for projections. 

After completing this work, the author became aware of independent forthcoming work of Barroero, Dill and K\"uhne in \cite{DillBarroeroKuhne} on related questions and, most recently, OpenAI also announced a proof of Zilber-Pink for abelian varieties in \cite{OAI:The-Abelian-Zilber-Pink-Conjecture-September-24-2026}. 

To prove Theorem \ref{OGEllipticConjecture}, we will follow a similar strategy to that of Boxall.
That is, in Section \ref{algebro-geometricPrelims}, we reduce Theorem \ref{OGEllipticConjecture} to the following.  

\begin{theorem}\label{OGTheoremReduced}
    Let $G$ be a simple abelian variety defined over $\mathbb{Q}^\mathrm{alg}$ and let $C_1, C_2 \subseteq G^3(\mathbb{C})$ be irreducible closed algebraic curves with at least one of the curves not defined over $\mathbb{Q}^\mathrm{alg}$. Suppose that there does not exist a proper algebraic subgroup $H \subseteq G^3(\mathbb{C})$ such that $C_1 \subseteq H$. Let $\mathcal{N} = \{n \in \mathbb{N} \ | \ [n]C_1 \subseteq C_2\}$. Then $\bigcup_{n \in \mathbb{N}\setminus\mathcal{N}}\{x \in C_1 \ | \ nx \in C_2\}$ is finite.
\end{theorem}  
Building on Zilber's work towards quasiminimality of $\mathbb{C}_\mathrm{exp}$ in \cite{Zilber2005}, Bays and Kirby in \cite{Bays_2018} provided a more general framework, known as $\Gamma$-fields, that allows one to study partial exponential fields and their analogues for semi-abelian varieties. Their framework provides the setting we will need in order to generalise Boxall's arguments. In Section \ref{gammaFields}, in the case of powers of simple abelian varieties, we introduce a slightly more flexible version of their construction. Namely, denoting the endomorphism ring of a simple abelian variety $G$ as $\mathrm{End}(G)$, a $\Gamma$-field is a field equipped with a divisible $\mathrm{End}(G)$-submodule of $(\mathbb{G}_{\rm{a}}^{\mathrm{dim}G} \times G)(F)$ where we require only that the underlying field contains the field generated by its associated module and we do not require this module to contain torsion (see Definition \ref{gammaFieldDefinition}). 
%
Moreover, we will refer to the $\Gamma$-field extensions, which play a role in our setting similar to that of finitely generated extensions in Bays and Kirby's framework, as exponentially finitely generated extensions. These are $\Gamma$-field extensions whose generators arise from their associated $\mathrm{End}(G)$-modules. 

In Section \ref{gammaClosureSubsection}, we reintroduce the $\Gamma$-closure $\Gamma\mathrm{cl}$, an analogue of exponential closure in our setting, also due to Bays and Kirby in \cite{Bays_2018} using our definition of $\Gamma$-fields and exponentially finitely generated extensions. In Section \ref{BKWSubsection}, we also prove the following analogue of the Bays-Kirby-Wilkie result in \cite{aSchanuelProperty}.

\begin{theorem}\label{BKWResultIntro}
    Let $G$ be a simple abelian variety of dimension $g$, let $F$ be a $\Gamma$-field with $F_\Gamma$ its associated $\mathrm{End}(G)$-module and let $n \in \mathbb{N}$. Suppose further that $F_\Gamma$ is the graph of a map denoted as $\mathrm{exp}_G$. Denoting $F_{\Gamma, 1}$ as the domain of $\mathrm{exp}_G$, let $a = (a_1, \ldots, a_n) \in F_{\Gamma, 1}$ and let $K \subseteq F$ such that $\Gamma\mathrm{cl}^F(K) = K$. Let $\lambda \in F$ such that $\lambda \not\in K$. Suppose that $\lambda a = (\lambda a_1, \ldots, \lambda a_n) \in F_{\Gamma, 1}$ and assume that $\mathrm{exp}_G(a) = (\mathrm{exp}_G(a_1), \ldots,\mathrm{exp}_G(a_n))$ is additively independent. Then \begin{align*}
        \mathrm{td}(\mathrm{exp}_G(a), \mathrm{exp}_G(\lambda a)/K, \lambda) \geq gn.
    \end{align*}
\end{theorem}


In Section \ref{otherPregeometrySubsection}, we discuss some further useful related pregeometries and relate these to $\Gamma\mathrm{cl}$. In Section \ref{finitenessResultChapter}, we conclude the paper by proving Theorem \ref{OGTheoremReduced} using results in Section \ref{otherPregeometrySubsection} together with Theorem \ref{BKWResultIntro}.

\section*{Acknowledgements}

The author would like to thank Vahagn Aslanyan, Gareth Jones, Martin Orr and Rodolphe Richard for helpful comments and conversations during the writing of this paper. The author is especially grateful to Gareth Boxall, who supervised the MSc project that this work was based on, for their patience, numerous discussions and comments during the preparation of this work. Particularly, the author would like to thank Gareth for suggesting the proof for Lemma \ref{algebraicAnalyticClosureRelatedToDefinableClosure}. 

\section{Some algebro-geometric preliminaries}\label{algebro-geometricPrelims}
In this section, we recall some standard algebro-geometric facts needed along with proving the reduction of Theorem \ref{OGEllipticConjecture} to Theorem \ref{OGTheoremReduced} following arguments of Boxall in \cite[Section 2]{boxall2024points}. Henceforth, let $G$ denote a simple abelian variety, namely a connected projective group variety that contains no proper abelian subvarieties, defined over $\mathbb{Q}^\mathrm{alg}$, denote $g := \mathrm{dim} G$, let $\mathrm{End}(G)$ denote its ring of endomorphisms and let $\mathrm{End}_0(G) := \mathrm{End(G)} \otimes_\mathbb{Z} \mathbb{Q}$ denote the endomorphism algebra of $G$.

Recall that, in $\mathbb{G}_{\rm{m}}^N(\mathbb{C})$, algebraic subgroups of codimension $r \in \{0, \ldots, N\}$ are given by the equations $X_1^{a_{i,1}} \cdot \dots \cdot X_N^{a_{i,n}} = 1$ with $a_{i, 1}, \ldots, a_{i, N} \in \mathbb{Z}$ for all $i \in \{1, \ldots, r\}$ such that $(a_{1,1}, \ldots, a_{1,N}), \ldots, (a_{r,1}, \ldots, a_{r,N})$ are $\mathbb{Q}$-linearly independent. It is well-known that powers of simple abelian varieties boast a similar property, but instead only have that each algebraic subgroup is contained in an abelian subvariety given by an analogous system of equations. The following can be deduced from various literature such as using similar arguments of the proof of \cite[Lemma 1]{Viada2003}. 
\begin{fact}\label{subgroupSatisfiesAdditiveRelations}
Let $N \geq 2$. Let $H \subseteq G^N(\mathbb{C})$ be an algebraic subgroup of codimension $gr$ for some $r \in \{1, \ldots, N\}$. Then there exists an algebraic subgroup $H'$ of $G^N(\mathbb{C})$ such that $H \subseteq H'$ and $H'$ is given as the solution set of $r$ equations of the form \begin{align*}
    \sum\limits_{j = 1}^{N}\lambda_{i, j}(x_j) = 0
\end{align*} with 
the $\lambda_{i, j} \in \mathrm{End}(G)$ for all $i \in \{1, \ldots, r\}$ such that $(\lambda_{1, 1}, ..., \lambda_{1, N}), ..., (\lambda_{r, 1}, \ldots, \lambda_{r, N})$ are left $\mathrm{End}_0(G)$-linearly independent. 
\end{fact}

As in the algebraic torus case, we will need to make use of Mordell-Lang, proven in a sequence of papers by Vojta, Faltings and Hindry, for curves in the case of powers of simple abelian varieties. 


\begin{fact}[Mordell-Lang \mbox{\cite{Hindry1988}}]\label{torsionTheorem}
    Let $C \subseteq G^N(\mathbb{C})$ be an irreducible closed algebraic curve for some $N \in \mathbb{N}$ with $N \geq 2$. Let $\Gamma$ be a finite rank subgroup of $G^N(\mathbb{C})$. If $C \cap \Gamma$ is infinite, then $C$ is a translate of an algebraic subgroup of $G^N(\mathbb{C})$. In particular, if $g > 1$, then $C \cap \Gamma$ is finite. 
\end{fact}

We now consider the analogues of \cite[Section 2, Corollary 2.4, Corollary 2.5 and Proposition 2.6]{boxall2024points}. The first is a consequence of Fact \ref{subgroupSatisfiesAdditiveRelations} and Fact \ref{torsionTheorem}, while the latter two are consequences of Corollary \ref{torsionCorollary}. Their proofs are essentially the same as in the toric case, so we will omit these. Following the strategy of Boxall, these will be useful later. Throughout this chapter, for any irreducible curve $C_1 \subseteq G^N(\mathbb{C})$, denote $[n]C_1 := \{nx \ | \ x \in C_1\}$.


\begin{corollary}\label{torsionCorollary}
    Let $N \geq 2$ and let $C \subseteq G^N(\mathbb{C})$ be an irreducible closed algebraic curve. Assume that at least one of the following conditions is satisfied: 
    \begin{enumerate}
        \item There is a non-torsion $x \in G^N(\mathbb{C})$ such that, for infinitely many $n \in \mathbb{N}$, $n\cdot x \in C$.
        \item There is a non-torsion $y \in G^N(\mathbb{C})$ such that, for infinitely many $n \in \mathbb{N}$, there exists $x \in C$ such that $n\cdot x = y$.
        \item $C$ contains infinitely many torsion points of $G^N(\mathbb{C})$.
    \end{enumerate} Then $\mathrm{dim} G = 1$ and $C$ is contained in an algebraic subgroup of $G^N(\mathbb{C})$ of dimension $1$. 
\end{corollary}

\begin{corollary}\label{indexIsFinite}
    Let $N \geq 2$ and $C_1, C_2 \subseteq G^N(\mathbb{C})$ be irreducible closed algebraic curves. Assume that there is no algebraic subgroup $H \subseteq G^N(\mathbb{C})$ of dimension $g$ such that $C_1 \subseteq H$. Let $\mathcal{N} = \{n \in \mathbb{N} \ | \ [n]C_1 \subseteq C_2\}$. Then $\mathcal{N}$ is finite. 
\end{corollary}

%

\begin{proposition}\label{propositionIForgot}
    Let $C_1, C_2 \subseteq G^N(\mathbb{C})$ be irreducible closed algebraic curves with $N \geq 3$. Assume that $C_1$ is not contained in a proper algebraic subgroup of $G^N(\mathbb{C})$. Let $\mathcal{N} = \{n \in \mathbb{N} \ | \ [n]C_1 \subseteq C_2\}$. If $C_1$ or $C_2$ is contained in a translate of an algebraic subgroup of $G^N(\mathbb{C})$ of dimension $g(N - 2)$, then $\bigcup_{n \in \mathbb{N}\setminus \mathcal{N}}\{x \in C_1 \ | \ nx \in C_2\}$ is finite.
\end{proposition}

For reasons that will become apparent in Section \ref{finitenessResultChapter}, we will need the following lemma in order to construct sequences to work with points in ultrapowers. 


\begin{lemma}\label{infiniteIndexes}
If $C_1, C_2 \subseteq G^N(\mathbb{C})$ are irreducible closed algebraic curves, with $N \geq 3$, and $\bigcup_{n \in \mathbb{N}\setminus\mathcal{N}} \{x \in C_1 \ | \ nx \in C_2\}$ infinite, where $\mathcal{N} = \{n \in \mathbb{N} \ | \ [n]C_1 \subseteq C_2\}$, then there exist infinitely many $n \in \mathbb{N}\setminus\mathcal{N}$ with the property that there exists $x \in C_1$ with $nx \in C_2$. 
\end{lemma}
\begin{proof}
    Assume that there existed only finitely many $n \in \mathbb{N}\setminus\mathcal{N}$ with the property that there exists $x \in C_1$ with $nx \in C_2$. Then there exists $n \in \mathbb{N}\setminus\mathcal{N}$ such that there are infinitely many $x \in C_1$ with $nx \in C_2$. As multiplication by $n$ is an endomorphism of $G^N(\mathbb{C})$, which we will denote by $\lambda_n: G^N(\mathbb{C}) \rightarrow G^N(\mathbb{C})$, it follows that $\lambda^{-1}_n(C_2)$ is a subvariety of $G^N(\mathbb{C})$. Hence, $\lambda^{-1}_n(C_2) \cap C_1$ is a subvariety of $C_1$. As there are infinitely many $x \in C_1$ such that $nx \in C_2$, we have that $\lambda^{-1}_n(C_2) \cap C_1$ must be a curve. As $C_1$ is irreducible, this implies that $\lambda^{-1}_n(C_2) \cap C_1 = C_1$ which contradicts that $n \not\in \mathcal{N}$.
%
%
\end{proof}

This last result of the section will allow us to reduce Theorem \ref{OGEllipticConjecture} to a simpler case where we need only consider curves in $G^3(\mathbb{C})$. The proof is essentially the same to that of Boxall in \cite{boxall2024points} and is given for the readers convenience. 

\begin{theorem}\label{thisIsIt}
    If Theorem \ref{OGTheoremReduced} is true, then Theorem $\ref{OGEllipticConjecture}$ holds.
\end{theorem}

\begin{proof}
    Let $N \geq 3$ and let $C_1, C_2 \subseteq G^N(\mathbb{C})$ be irreducible closed algebraic curves such that $C_1$ is not contained in a $g$-dimensional algebraic subgroup of $G^N(\mathbb{C})$ and $C_1 \cup C_2$ is not contained in a $2g$-dimensional algebraic subgroup of $G^N(\mathbb{C})$. We will argue by contradiction by considering several cases. Assume that $\bigcup_{n \in \mathbb{N}\setminus\mathcal{N}} \{x \in C_1 \ | \ nx \in C_2\}$ is infinite.

    \underline{Case 1}: Suppose that there is a $2g$-dimensional algebraic subgroup $H \subseteq G^N(\mathbb{C})$ such that $C_1 \subseteq H$. We break this case further into two subcases.

    \underline{Subcase 1.1}: Assume that there are infinitely many $y \in C_2$ such that there exist $x \in C_1$ and $n \in \mathbb{N}\setminus\mathcal{N}$ such that $nx = y$. Denote $S$ as the set of all such $y \in C_2$. Since $C_1 \subseteq H$ and $H$ is an algebraic subgroup of $G^N(\mathbb{C})$ it follows that $S \subseteq H$. We have that $S \subseteq C_2 \cap H$ so that $C_2 \cap H$ is infinite. We must have that $C_2 \cap H$ is a subvariety of $C_2$ of dimension $1$. Since $C_2$ is irreducible and also of dimension $1$, we deduce that $C_2 \cap H = C_2$ so that $C_2 \subseteq H$. This contradicts our assumption that $C_1 \cup C_2$ cannot be contained in a $2g$-dimensional algebraic subgroup of $G^N(\mathbb{C})$. 

    \underline{Subcase 1.2}: Now, assume the contrary so that there are only finitely many $y \in C_2$ such that there exists $x \in C_1$ and $n \in \mathbb{N}\setminus\mathcal{N}$ such that $nx = y$. Then there exists a $y \in C_2$ such that, for infinitely many $x \in C_1$, there is some $n \in \mathbb{N}\setminus\mathcal{N}$ such that $nx = y$. If $y$ is a torsion point, then $C_1$ would have infinitely many torsion points which, by Corollary \ref{torsionCorollary}, implies that $g = 1$ and $C_1$ is contained in a $g$-dimensional algebraic subgroup. If $y$ were non-torsion, then again, by Corollary \ref{torsionCorollary}, $g = 1$ and $C_1$ would be contained in a $g$-dimensional algebraic subgroup. In either case, this is a contradiction to our assumptions on $C_1$.

    \underline{Case 2}: Assume that $C_1$ is not contained in a $2g$-dimensional algebraic subgroup. 
    
    \vspace{10pt}
    \underline{Claim}: Since at least one of $C_1$ or $C_2$ is not defined over $\mathbb{Q}^\mathrm{alg}$, there exists a projection $\pi: G^N(\mathbb{C}) \rightarrow G^3(\mathbb{C})$ such that $\pi(C_1)$ is not contained in a proper algebraic subgroup of $G^3(\mathbb{C})$ and that at least one of $\pi(C_1)^\mathrm{Zar}$ or  $\pi(C_2)^\mathrm{Zar}$ is not defined over $\mathbb{Q}^\mathrm{alg}$.

    \vspace{10pt}

    Indeed, suppose that no such projection existed. 
    Let $\pi': G^N(\mathbb{C}) \rightarrow G^k(\mathbb{C})$ be a projection such that $\pi'(C_1)$ is not contained in a proper algebraic subgroup of $G^k(\mathbb{C})$. We may assume that $k$ is maximal so that $k \geq 3$ as $C_1$ is not contained in a $2g$-dimensional algebraic subgroup. 
    We now claim that $\bigcup_{n \in \mathbb{N}\setminus\mathcal{N}}\{y \in C_2 \ | \ \text{ there exists } x \in C_1 \text{ such that } nx = y\}$ is infinite. To prove this, assume the contrary so that there exists $y \in C_2$ such that there are infinitely many $x \in C_1$ such that there exists an $n \in \mathbb{N}\setminus\mathcal{N}$ where $nx = y$. If $y$ is not a torsion point, then, by Corollary \ref{torsionCorollary}, $g = 1$ and $C_1$ is contained in an algebraic subgroup of $G^N(\mathbb{C})$ of dimension $g$. If $y$ is a torsion point, then by Corollary \ref{torsionCorollary}, $g = 1$ and $C_1$ is again contained in an algebraic subgroup of $G^N(\mathbb{C})$ of dimension $g$. In both cases, we reach a contradiction. So, $\bigcup_{n \in \mathbb{N}\setminus\mathcal{N}}\{y \in C_2 \ | \ \text{ there exists } x \in C_1 \text{ such that } nx = y\}$ is infinite. We now claim that $y$ has additive dimension zero over $\pi'(y)$ for all $y \in C_2$. If $k = N$, then the claim is immediate. Assume that $k < N$ and assume, for simplicity, that $\pi'(y_1, \ldots, y_N) = (y_1, \ldots, y_k)$ for any $(y_1, \ldots, y_N) \in G^N$. 
    Let $j > k$ and let $\pi'': G^N(\mathbb{C}) \rightarrow G^{k+1}(\mathbb{C})$ be the projection that, which we also may assume for simplicity, sends $(y_1, \ldots, y_N)$ to $(y_1, \ldots, y_k, y_j)$. By the maximality of $k$, $\pi''(C_1)$ is contained in a proper algebraic subgroup $H''$ of $G^{k + 1}(\mathbb{C})$. It follows that there is a proper algebraic subgroup $H$ of $G^N(\mathbb{C})$ such that $\pi''(H) = H''$. Let $x \in C_1$ such that there exists an $n \in \mathbb{N}\setminus\mathcal{N}$ where $nx \in C_2$. As $H$ is an algebraic group and $x \in H$, it follows that, for any $n \in \mathbb{N}$, $nx \in H$. As $\bigcup_{n \in \mathbb{N}\setminus\mathcal{N}}\{y \in C_2 \ | \ \text{ there exists } x \in C_1 \text{ such that } nx = y\}$ is infinite, it follows that there are infinitely many points in $H \cap C_2$. As in the proof of Lemma \ref{infiniteIndexes}, since $C_2$ is of dimension 1 and is irreducible, we deduce that $C_2 \subseteq H$ so that $\pi''(C_2) \subseteq H''$. Hence, the additive dimension of $(y_1, \ldots, y_k, y_j)$ over $(y_1, \ldots, y_k)$ is zero. As $j$ was arbitrary, we deduce that $(y_1, \ldots, y_N)$ has additive dimension zero over $(y_1, \ldots, y_k)$, thus proving our second claim. Since, by assumption, $\pi'(C_1)$ and $\pi'(C_2)$ are defined over $\mathbb{Q}^\mathrm{alg}$, 
    it follows that $C_1$ and $C_2$ are defined over $\mathbb{Q}^\mathrm{alg}$. This is a contradiction.  \qed
    \vspace{10pt}

    Now, let $\pi: G^N(\mathbb{C}) \rightarrow G^3(\mathbb{C})$ be a projection such that $\pi(C_1)$ is not contained in a proper algebraic subgroup of $G^3(\mathbb{C})$ and that, letting $C_1' := \pi(C_1)^\mathrm{Zar}$ and $C_2' := \pi(C_2)^\mathrm{Zar}$, at least one of $C_1'$ or $C_2'$ is not defined over $\mathbb{Q}^\mathrm{alg}$. 
    Since $C_1$ and $C_2$ are irreducible, $C_1'$ and $C_2'$ are both irreducible. We have that $C_1'$ is a curve. Indeed, suppose that this were not true. We have that $\bigcup_{n \in \mathbb{N}\setminus\mathcal{N}} \{x \in C_1' \ | \ nx \in C_2'\}$ is non-empty since $\bigcup_{n \in \mathbb{N}\setminus\mathcal{N}} \{x \in C_1 \ | \ nx \in C_2\}$ is non-empty. Then $C_1'$ would be a single point, say $x \in C_1'$, as it is irreducible and, by Lemma \ref{infiniteIndexes}, $nx \in C_2'$ for infinitely many $n \in \mathbb{N}\setminus\mathcal{N}$. If $x$ is not a torsion point, then Corollary \ref{torsionCorollary} gives that $g = 1$ and $C_2'$ is contained in an algebraic subgroup of dimension $g$. This implies that $C_1'$ is also contained in a proper algebraic subgroup which is a contradiction. If $x$ were a torsion point, then it would be contained in the algebraic subgroup defined by $\sum_{i = 1}^N\lambda_m(x_i) = 0$ where $\lambda_m$ is the endomorphism of $G$ corresponding to multiplying by some $m \in \mathbb{N}$. This is also a contradiction. We now consider various subcases as before. Let $\mathcal{N}' = \{n \in \mathbb{N} \ | \ [n]C_1' \subseteq C_2'\}$.

    \underline{Subcase 2.1}: Assume that $C_2'$ is finite. Since $C_2'$ is irreducible, it consists of a single point, say $y \in C_2'$. We consider two subcases.

    \underline{Subcase 2.1.1}: Assume that $\bigcup_{n \in \mathbb{N}\setminus\mathcal{N}'}\{x \in C_1' \ | \ nx \in C_2'\}$ is finite. If $\bigcup_{n \in \mathbb{N}\setminus\mathcal{N}'}\{x \in C_1 \ | \ nx \in C_2\}$ is infinite, then there is $x \in C_1'$ such that $\pi^{-1}(\{x\}) \cap C_1$ is infinite. Thus, $\pi^{-1}(\{x\}) \cap C_1$ is a curve in $C_1$. It follows from previous arguments that $\pi^{-1}(\{x\}) \cap C_1 = C_1$. This contradicts that $C_1'$ is infinite. Thus, $\bigcup_{n \in \mathbb{N}\setminus\mathcal{N}'}\{x \in C_1 \ | \ nx \in C_2\}$ is finite. By Corollary \ref{indexIsFinite}, $\mathcal{N}'$ is finite. 
    We conclude that $\bigcup_{n \in \mathbb{N}\setminus\mathcal{N}}\{x \in C_1 \ | \ nx \in C_2\}$ is finite.

    \underline{Subcase 2.1.2}: Assume that $\bigcup_{n \in \mathbb{N}\setminus\mathcal{N}'}\{x \in C_1' \ | \ nx \in C_2'\}$ is infinite. If $y$ is a non-torsion point then, by Corollary \ref{torsionCorollary}, $g = 1$ and $C_1'$ is contained in a $g$-dimensional algebraic subgroup which implies that $C_1$ is also contained in an algebraic subgroup of codimension $2g$. If $y$ is a torsion point, then $C_1'$ contains infinitely many torsion points so that, by Corollary \ref{torsionCorollary}, $g = 1$ and $C_1'$ is contained a $g$-dimensional algebraic subgroup. This, again, implies that $C_1$ is contained an algebraic subgroup of codimension $2g$. In both cases, using Proposition \ref{propositionIForgot}, this is a contradiction to our assumptions.

    \underline{Subcase 2.2}: Assume that $C_2'$ is infinite. Then $C_1'$ and $C_2'$ are both curves. By assumption, $\bigcup_{n \in \mathbb{N}\setminus\mathcal{N}'}\{x \in C_1' \ | \ nx \in C_2'\}$ is finite. We deduce that $\bigcup_{n \in \mathbb{N}\setminus\mathcal{N}}\{x \in C_1 \ | \ nx \in C_2\}$ is finite by similar arguments to that of subcase 2.1.1. 
\end{proof}

\section{Certain related pregeometries and a Bays-Kirby-Wilkie-type result}\label{BKWSection}

\subsection{$\Gamma$-fields}\label{gammaFields} 
We will henceforth assume that all fields are of characteristic $0$. Recall that $G$ is a simple abelian variety defined over a number field $\mathcal{K}_0$ where we denote $g := \mathrm{dim} G$. We will assume that $\mathrm{End}(G)$ and $\mathrm{End}_0(G)$ are defined over $\mathcal{K}_0$. 
Set $F$ as an algebraically closed field that is a field extension of $\mathcal{K}_0$ and consider $G(F)$ as an $\mathrm{End}(G)$-module equipped with $G(F)$'s group operation as addition and $\mathrm{End}(G)$ acting on $G(F)$ as scalar multiplication. Since each endomorphism $\lambda \in \mathrm{End}(G)$ lifts to a unique $F$-linear endomorphism   $\alpha_\lambda$ of $\mathbb{G}_{\rm{a}}(F)$, we may also consider $\mathbb{G}_{\rm{a}}(F)$ as an $\mathrm{End}(G)$-module equipped with the usual addition and scalar multiplication given by the action of $\alpha_\lambda$ for each $\alpha \in \mathrm{End}(G)$. We will identify $TG$ as $\mathbb{G}_{\rm{a}}^g \times G$ considered with the natural $\mathrm{End}(G)$-module structure. Note that, given an algebraic variety $V$ and some $A \subseteq V(F)$, we call the smallest field $K \subseteq F$, such that $A \subseteq V(K)$, the field generated by coordinates in $A$.

\begin{definition}\label{gammaFieldDefinition}
    A \emph{$\Gamma$-field} (with respect to the $\mathrm{End}(G)$-module $TG(F)$)
    is a pair $(K, K_\Gamma)$ where $K_\Gamma$ is a divisible $\mathrm{End}(G)$-submodule of $TG(F)$ and $K$ is a field extension of $\mathcal{K}_0$ such that 
    $K$ contains the field generated by 
    the coordinates of $K_\Gamma$. We will usually write $(K, K_\Gamma)$ simply as $K$ and write $\pi_1: TG(K) \rightarrow \mathbb{G}_{\rm{a}}^g(K)$ and $\pi_2: TG(K) \rightarrow G(K)$ as the natural projections.
\end{definition}
We will, henceforth, consider $F$ as a $\Gamma$-field with respect to $TG(F)$ equipped with some divisible $\mathrm{End}(G)$-submodule of $TG(F)$. 
\begin{definition}
    An \emph{extension} of a $\Gamma$-field $K$ is a $\Gamma$-field $L$ such that 
    $K_\Gamma \subseteq L_\Gamma$, $K$ is a subfield of $L$ and $K_\Gamma$ contains all torsion of $L_\Gamma$. 
\end{definition}

Following the arguments of \cite[Lemma 3.14]{Bays_2018}, this next definition is well-defined.
\begin{definition}
            Let $L$ be a $\Gamma$-field and let $\{K_i \ | \ i \in I\}$ be a family of $\Gamma$-subfields of $L$. 
        We define $\bigwedge_{i \in I} K_i$ to be the $\Gamma$-field with underlying field
        $\bigcap_{i \in I} K_i$ and $\mathrm{End}(G)$-module $\bigcap_{i \in I} {K_{i}}_\Gamma$.
\end{definition}
This permits us to give the following definition. 

\begin{definition}
    Let $L$ be a $\Gamma$-field, let $X \subseteq L$ and let $Y \subseteq L_\Gamma$. Let $\mathcal{B}_{X, Y}$ be the set of all {$\Gamma$-subfields} $K$ of $L$ 
    such that $X \subseteq K$ and $Y \subseteq K_\Gamma$. We say that $\bigwedge \mathcal{B}_{X, Y}$ is the $\Gamma$-subfield \emph{generated by $X$ and $Y$} and denote it by $\langle X, Y\rangle_{L}$ where we usually omit the subscript $L$ when it is understood. 
\end{definition}

%

Deviating slightly from Bays and Kirby's approach in \cite{Bays_2018}, we will mainly consider finitely generated $\Gamma$-field extensions generated only using elements from $\pi_2(F_\Gamma)$ which leads us to the following definition.


\begin{definition}\label{OGfinitelyGeneratedExtension}
    Let $K$ be a $\Gamma$-field and let $L$ be a $\Gamma$-field extension of $K$. We 
    say that $L$ is \emph{finitely generated over $K$} if there exist a finite set $X \subseteq L$ and a finite set $Y \subseteq L_\Gamma$ such that $L = \langle K \cup X \cup Z, K_\Gamma \cup Y \rangle_L$ where $Z$ is the field generated by the coordinates of $Y$. We say that $L$ is  \emph{exponentially finitely generated over $K$} if $X = \varnothing$.
\end{definition}

\subsection{$\Gamma$-closure}\label{gammaClosureSubsection}
We begin this section by introducing some useful finitary closure operators. Given a set $S$, recall that a map $\mathrm{cl}: \mathcal{P}(S) \rightarrow \mathcal{P}(S)$ is a \emph{finitary closure operator on $S$} if it is extensive ($A \subseteq \mathrm{cl}(A)$ for all $A \subseteq S$), monotone (for all $A, B \subseteq S$, $A \subseteq B \implies \mathrm{cl}(A) \subseteq \mathrm{cl}(B))$), idempotent ($\mathrm{cl}(\mathrm{cl}(A)) = \mathrm{cl}(A)$ for all $A \subseteq S$) and satisfies the finite character property (for all $A \subseteq S$, $a \in \mathrm{cl}(A) \implies a \in \mathrm{cl}(A_0)$ for some finite $A_0 \subseteq A$). Some examples of finitary closure operators are as follows.
    \begin{itemize}
        \item $\mathrm{span}^{V_D}: \mathcal{P}(V_D) \rightarrow \mathcal{P}(V_D)$, the span of a subset, where $V_D$ is a left $D$-vector space over a division algebra $D$.
        \item $\mathrm{acl}^K: \mathcal{P}(K) \rightarrow \mathcal{P}(K)$, the relative algebraic closure, where $K$ is a field.
        \item $\mathrm{dcl}^R: \mathcal{P}(K) \rightarrow \mathcal{P}(K)$, the definable closure with $\mathrm{dcl}^R(A)$ defined to be the union of all sets of singletons definable in $R$ with parameters in $A$, where $R$ is an o-minimal structure.
    \end{itemize}
When using these, we will often drop the subscripts when the context is clear. The above examples are  particularly special examples of finitary closure operators, known as \emph{pregeometries (or matroids)}, which satisfy the Steinitz-exchange property (for all $a, b \in S$ and $A \subseteq S$, $a \in \mathrm{cl}(A\cup \{b\}) \wedge a \not\in \mathrm{cl}(A)\implies b \in \mathrm{cl}(A\cup \{a\})$). Every pregeometry $\mathrm{cl}$ on a set $S$ induces a well-defined notion of independence and dimension. In particular, a subset $A \subseteq S$ is said to be $\mathrm{cl}$-independent over $B \subseteq S$ if, for every $x \in A$, we have $x \not\in \mathrm{cl}(B \cup (A\setminus\{x\}))$. The dimension function $\mathrm{dim}_\mathrm{cl}(A/B)$ for $A, B \subseteq S$ is then defined to be the cardinality of any maximal $\mathrm{cl}$-independent subset of $A$ over $B$ which is well-defined by the exchange property.
In this paper, we will primarily make use of transcendence degree $\mathrm{tr.deg} := \mathrm{dim}_\mathrm{acl}$ and linear dimension $\mathrm{ldim} := \mathrm{dim}_{\mathrm{span}}$.

Now, to capture the interaction between algebraic independence and $\Gamma$-field extensions, we reintroduce the predimension function due to Hrushovski (see \cite{HRUSHOVSKI}) following the presentation and methods of \cite[Sections 4.1 and 4.2]{Bays_2018} to define the finitary closure operator $\Gamma\mathrm{cl}$. 

Given $\Gamma$-fields $K$ and $L$ such that $L$ is a $\Gamma$-field extension of $K$, we can associate a natural vector space to this extension. Indeed, the quotient space $L_\Gamma/K_\Gamma$ is divisible and, by definition of $\Gamma$-field extensions, is torsion-free, which implies that this is a finite dimensional $\mathrm{End}_0(G)$-vector space, allowing the following. 
%
%
\begin{definition}
    Let $K$ be a $\Gamma$-field. For any $\Gamma$-field $L$ 
    that is finitely generated over $K$, let \begin{align*}
        \delta(L/K) := \mathrm{tr.deg}(L/K) - d \mathrm{ldim}_{\mathrm{End}_0(G)}(L_\Gamma/K_\Gamma).
    \end{align*} 
\end{definition}
As a convention, for any $\Gamma$-subfield $K \subseteq F$ and $(a_1, \ldots, a_n) \in F_\Gamma^N$ for some $n \in \mathbb{N}$, we denote $\delta(a_1, \ldots, a_n/K) := \delta(X/K)$ 
where $X = \langle K \cup X_{\{a_1, \ldots, a_n\}}, K_\Gamma \cup \{a_1, \ldots, a_n\}\rangle_F$ with $X_{\{a_1, \ldots, a_n\}}$ denoting the subfield of $F$ generated by $\{a_1, \ldots, a_n\}$.

\begin{definition}
    A $\Gamma$-subfield $A$ of $F$ is $\Gamma$-closed in $F$, written
$A \triangleleft_\mathrm{cl} F$, if, for any $A \subseteq B \subseteq F$ with $B$ exponentially finitely generated over $A$ and $\delta(B/A) \leq 0$, we have $B = A$ as $\Gamma$-fields. 
\end{definition}

We can now in a position to define our needed closure operator. 
\begin{definition}
\label{gammaClosureDefinition}
    Let $A \subseteq F$. The \emph{$\Gamma$-closure of $G$ in $F$} is defined to be the smallest $\Gamma$-closed $\Gamma$-subfield containing $G$, 
    \begin{align*}
        \Gamma \mathrm{cl}^F(A) := \bigwedge \{B \triangleleft_\mathrm{cl} F \ | \ A \subseteq B\}.
    \end{align*} This induces a map $\Gamma\mathrm{cl}^F: \mathcal{P}(F) \rightarrow \mathcal{P}(F)$.
\end{definition}

Using similar arguments to \cite[Proposition 4.14]{Bays_2018}, we obtain the following. 
\begin{lemma}
    $\Gamma\mathrm{cl}^F: \mathcal{P}(F) \rightarrow \mathcal{P}(F)$ is a finitary closure operator. 
\end{lemma}

\subsection{Bays-Kirby-Wilkie}\label{BKWSubsection}
As in \cite[Theorem 3.7]{boxall2024points}, we now prove an analogue in our setting of the Bays-Kirby-Wilkie result in \cite[Theorem 1.2]{aSchanuelProperty}. Suppose further that $F_\Gamma$ is the graph of a map which will denote by $\mathrm{exp}_G$. As a convention, we will denote $(\mathrm{exp}_G(a_1), \ldots, \mathrm{exp}_G(a_n))$ as $\mathrm{exp}_G(a)$, for any $n \in \mathbb{N}$ and $a = (a_1, \ldots, a_n) \in F^n$, when $\mathrm{exp}_G(a_i)$ is defined for all $i \in \{1, \ldots, n\}$.

\begin{theorem}\label{additiveIndependenceTheorem}
    Let $N \in \mathbb{N}$. Let $a = (a_1, \ldots, a_n) \in \pi_1(F_\Gamma)^n$ and let $K \subseteq F$ such that $\Gamma\mathrm{cl}^F(K) = K$. Let $\lambda \in F$ such that $\lambda \not\in K$ and $\lambda a = (\lambda a_1, \ldots, \lambda a_n) \in \pi_1(F_\Gamma)^n$. Assume that $\mathrm{exp}_G(a) = (\mathrm{exp}_G(a_1), \ldots,\mathrm{exp}_G(a_n))$ is additively independent. Then \begin{align*}
        \mathrm{tr.deg}(\mathrm{exp}_G(a), \mathrm{exp}_G(\lambda a)/K, \lambda) \geq gn.
    \end{align*}
\end{theorem}


\begin{proof}
Let $b = (a, \lambda a)$. Since $\lambda \not\in K$ and $K$ is $\Gamma$-closed, it follows that 
\begin{align*}
        \mathrm{tr.deg}(b, \mathrm{exp}_G(b)/K) -g\mathrm{ldim}_{\mathrm{End}_0(G)}(b, \mathrm{exp}_G(b)/K_\Gamma) \geq 1
    \end{align*} which gives that \begin{align*}
        \mathrm{tr.deg}(b/K) + \mathrm{tr.deg}(\mathrm{exp}_G(b)/K, b) - g\mathrm{ldim}_{\mathrm{End}_0(G)}(b, \mathrm{exp}_G(b)/K_\Gamma) \geq 1.
    \end{align*}
    We have \begin{align*}
        \mathrm{tr.deg}(b/K, \lambda) \leq g\mathrm{ldim}_{\mathrm{End}_0(G)(\lambda)}(b/K)
    \end{align*} where $\mathrm{End}_0(G)(\lambda)$ denotes the division algebra generated by $\mathrm{End}_0(G)$ and the central element $\lambda$. We also have that \begin{align*}
        \mathrm{tr.deg}(\mathrm{exp}_G(b)/K, b) \leq \mathrm{tr.deg}(\mathrm{exp}_G(b)/K, \lambda).
    \end{align*} It follows that \begin{align*}
        \mathrm{tr.deg}(\mathrm{exp}_G(b)/K, \lambda) + g\mathrm{ldim}_{\mathrm{End}_0(G)(\lambda)}(b/K) - g\mathrm{ldim}_{\mathrm{End}_0(G)}(b, \mathrm{exp}_G(b)/K_\Gamma) \geq 0.
    \end{align*} We now consider $\mathrm{ker} = \{c \in F \ | \ \mathrm{exp}_G(c) = 0_G\}$ and let $\mathrm{ker}_\Gamma = \{(c, \mathrm{exp}_G(c)) \ | \ c \in \mathrm{ker}\}$. By finite character of $\mathrm{span}$, there exists $(l, \mathrm{exp}_G(l)) \in K_\Gamma^n$ for some $n \in \mathbb{N}$ such that \begin{align*}
        g\mathrm{ldim}_{\mathrm{End}_0(G)}(b, \mathrm{exp}_G(b)/K_\Gamma) = g\mathrm{ldim}_{\mathrm{End}_0(G)}(b, \mathrm{exp}_G(b)/\langle \{(l, \mathrm{exp}_G(l))\} \cup \mathrm{ker}_\Gamma\rangle).
    \end{align*} 
    Let $K'$ be the $End(G)$-module generated by $\{l\} \cup \mathrm{ker}$. 
    We have $K' \subseteq K$ so that \begin{align*}
        \mathrm{tr.deg}(\mathrm{exp}_G(b)/K, \lambda) + g\mathrm{ldim}_{\mathrm{End}_0(G)(\lambda)}(b/K') - g\mathrm{ldim}_{\mathrm{End}_0(G)}(b, \mathrm{exp}_G(b)/K_\Gamma) \geq 0.
    \end{align*} We now have \begin{align*}
        & \ \ \ \ g\mathrm{ldim}_{\mathrm{End}_0(G)}(b/\mathrm{ker}) - g\mathrm{ldim}_{\mathrm{End}_0(G)}(b, \mathrm{exp}_G(b)/K_\Gamma) \\
        &= g\mathrm{ldim}_{\mathrm{End}_0(G)}(b, \mathrm{exp}_G(b)/\mathrm{ker}_\Gamma) - g\mathrm{ldim}_{\mathrm{End}_0(G)}(b, \mathrm{exp}_G(b)/\langle \{(l, \mathrm{exp}_G(l))\} \cup \mathrm{ker}_\Gamma\rangle) \\
        &= (d\mathrm{ldim}_{\mathrm{End}_0(G)}(b, \mathrm{exp}_G(b), l, \mathrm{exp}_G(l)/\mathrm{ker}_\Gamma) - g\mathrm{ldim}_{\mathrm{End}_0(G)}(l, \mathrm{exp}_G(l)/\langle\{b, \mathrm{exp}_G(b)\} \cup \mathrm{ker}_\Gamma\rangle)) \\ & \ \ \ \  - (d\mathrm{ldim}_{\mathrm{End}_0(G)}(b, \mathrm{exp}_G(b), l, \mathrm{exp}_G(l)/\mathrm{ker}_\Gamma) - g\mathrm{ldim}_{\mathrm{End}_0(G)}(l, \mathrm{exp}_G(l)/\mathrm{ker}_\Gamma)) \\
        &= g\mathrm{ldim}_{\mathrm{End}_0(G)}(l, \mathrm{exp}_G(l)/\mathrm{ker}_\Gamma) - g\mathrm{ldim}_{\mathrm{End}_0(G)}(l, \mathrm{exp}_G(l)/\langle \{(b, \mathrm{exp}_G(b))\} \cup \mathrm{ker}_\Gamma\rangle) \\
        &= g\mathrm{ldim}_{\mathrm{End}_0(G)}(l/\mathrm{ker}) - g\mathrm{ldim}_{\mathrm{End}_0(G)}(l, \mathrm{exp}_G(l)/\langle \{(b, \mathrm{exp}_G(b))\} \cup \mathrm{ker}_\Gamma\rangle) 
        \end{align*}
        where the third line above follows from the addition formula. 
        We next claim that \begin{align*}
             g\mathrm{ldim}_{\mathrm{End}_0(G)}(l/\mathrm{ker}) =  g\mathrm{ldim}_{\mathrm{End}_0(G)(\lambda)}(l/\mathrm{ker}).
        \end{align*} Arguing for a contradiction, suppose that this was not the case. There then exists a non-trivial dependence \begin{align}\label{inclusion2}
            \sum_{i = 1}^r p_i(\lambda)u_i \in \langle\mathrm{ker}\rangle_{\mathrm{End}_0(G)(\lambda)}
        \end{align} where $r = \mathrm{ldim}_{\mathrm{End}_0(G)}(l/\mathrm{ker})$ and $u_1, \ldots, u_r \in \langle l, \mathrm{ker}\rangle_{\mathrm{End}_0(G)}$ are chosen such that their images form a basis for $\langle l, \mathrm{ker}\rangle_{\mathrm{End}_0(G)}/\langle\mathrm{ker}\rangle_{\mathrm{End}_0(G)}$ and $p_i \in \mathrm{End}_0(G)(X)$ for all $i \in \{1, \ldots, r\}$. It follows that \begin{align*}
            \sum_{i = 1}^r p_i(\lambda)u_i + \sum_{j= 1}^s q_j(\lambda) v_j = 0
        \end{align*} where $s \in \mathbb{N}$, $v_1, \ldots, v_s \in \mathrm{ker}$ and $q_j \in \mathrm{End}_0(G)(X)$ for all $j \in \{1, \ldots, s\}$. Clearing denominators, we may assume that 
        $p_i, q_j \in \mathrm{End}_0(G)[X]$ for all $i \in \{1, \ldots, r\}$ and all $j \in \{1, \ldots, s\}$.
        We then have \begin{align}\label{equation1}
            \sum^t_{k = 1} c_k\lambda^k = 0
        \end{align} where $t \in \mathbb{N}$ and $c_k$ is a linear combination of the $u_i$'s and $v_j$'s for all $k \in \{1, \ldots, t\}$. Since $\mathrm{End}_0(G)$ is defined over a finite extension of $\mathbb{Q}$, it follows that $\mathrm{End}_0(G)$ is defined over $K$. But $\lambda \not\in K$ which, by (\ref{equation1}), implies that $c_k = 0$ for all $k \in \{1, \ldots, t\}$. This implies that $p_i, q_j \in \mathrm{End}_0(G)$ for all $i \in \{1, \ldots, r\}$ and $j \in \{1, \ldots, s\}$. By our choice of $u_i$'s, we must have that $p_i = 0$ for all $i \in \{1, \ldots, r\}$, contradicting that (\ref{inclusion2}) is non-trivial. It follows that 
        \begin{align*}
        &d\mathrm{ldim}_{\mathrm{End}_0(G)}(l/\mathrm{ker}) - g\mathrm{ldim}_{\mathrm{End}_0(G)}(l, \mathrm{exp}_G(l)/\langle \{(b, \mathrm{exp}_G(b))\} \cup \mathrm{ker}_\Gamma\rangle) \\
        &=  g\mathrm{ldim}_{\mathrm{End}_0(G)(\lambda)}(l/\mathrm{ker}) - g\mathrm{ldim}_{\mathrm{End}_0(G)}(l, \mathrm{exp}_G(l)/\langle \{(b, \mathrm{exp}_G(b))\} \cup \mathrm{ker}_\Gamma\rangle) \\
        &= g\mathrm{ldim}_{\mathrm{End}_0(G)(\lambda)}(l/\mathrm{ker}) - g\mathrm{ldim}_{\mathrm{End}_0(G)}(l/\{b\} \cup \mathrm{ker}) \\
        &\leq g\mathrm{ldim}_{\mathrm{End}_0(G)(\lambda)}(l/\mathrm{ker}) - g\mathrm{ldim}_{\mathrm{End}_0(G)(\lambda)}(l/\{b\} \cup \mathrm{ker}) \\
        &= (d\mathrm{ldim}_{\mathrm{End}_0(G)(\lambda)}(l, b/\mathrm{ker}) - g\mathrm{ldim}_{\mathrm{End}_0(G)(\lambda)}(b/l \cup \mathrm{ker}))  \\
        & \ \ \ \ - (d\mathrm{ldim}_{\mathrm{End}_0(G)(\lambda)}(l, b/\mathrm{ker}) - g\mathrm{ldim}_{\mathrm{End}_0(G)(\lambda)}(b/\mathrm{ker})) \\
        &= g\mathrm{ldim}_{\mathrm{End}_0(G)(\lambda)}(b/\mathrm{ker}) - g\mathrm{ldim}_{\mathrm{End}_0(G)(\lambda)}(b/K').
    \end{align*} where the second-last line above follow from the addition formula. 
    We deduce that \begin{align*}
        \mathrm{ldim}_{\mathrm{End}_0(G)}(b, \mathrm{exp}_G(b)/K_\Gamma) - \mathrm{ldim}_{\mathrm{End}_0(G)(\lambda)}(b/K') \geq \mathrm{ldim}_{\mathrm{End}_0(G)}(b/\mathrm{ker}) - \mathrm{ldim}_{\mathrm{End}_0(G)(\lambda)}(b/\mathrm{ker})
    \end{align*} so that \begin{align*}
        \mathrm{tr.deg}(\mathrm{exp}_G(b)/K, \lambda) + g\mathrm{ldim}_{\mathrm{End}_0(G)(\lambda)}(b/\mathrm{ker}) - g\mathrm{ldim}_{\mathrm{End}_0(G)}(b/\mathrm{ker}) \geq 0.
    \end{align*} We need only show that $\mathrm{ldim}_{\mathrm{End}_0(G)}(b/\mathrm{ker}) - \mathrm{ldim}_{\mathrm{End}_0(G)(\lambda)}(b/\mathrm{ker}) \geq n$. This is done analagously to that at the bottom of \cite[p.~4]{aSchanuelProperty}, where we replace $\mathbb{Q}$ with $\mathrm{End}_0(G)$ as appropriate. We note that our $a, b, n$ will correspond respectively to $x, z, n$ in the relevent sections of \cite{aSchanuelProperty}. 
\end{proof}

\subsection{Other related pregeometries}\label{otherPregeometrySubsection}
We now introduce our o-minimal setting. 
As defined in Section \ref{gammaFields}, suppose that $\mathcal{K}_0 = \mathbb{Q}^\mathrm{alg}$. Now, let $\mathrm{exp}_G$ denote the usual exponential map of $G$. It is well-known that 
$\mathrm{exp}_G: \mathbb{G}_{\rm{a}}^g(\mathbb{C}) \rightarrow G(\mathbb{C})$, when restricted to a suitable fundamental domain $\mathbb{F}_G$ of $\mathrm{exp}_G$ centered at $(0, \ldots, 0) \in \mathbb{G}_{\rm{a}}^g(\mathbb{C})$, is definable in $\mathbb{R}_{\mathrm{an}, \mathrm{exp}}$. Let $\mathcal{R}$ be a reduct of $\mathbb{R}_{\mathrm{an}, \mathrm{exp}}$ with countable language where $\mathrm{exp}_G|_{\mathbb{F}_G}$ is definable. Replacing $\mathcal{R}$ with an appropriate elementary extension, we may further assume that $\mathcal{R}$ is $\kappa$-saturated for some infinite cardinal $\kappa > 2^{\aleph_0}$ by replacing it with a suitable elementary extension. Let $\mathcal{I}$ and $\mathcal{J}$ denote the infinitesimal elements of $\mathcal{R}$ and $\tilde{\mathcal{R}}$, respectively, and let $\Tilde{\mathcal{R}}$ be an ultrapower of $\mathcal{R}$ with respect to a non-prinicipal ultrafilter $\mathcal{U}$ on $\mathbb{N}$. Let $\Tilde{\mathbb{N}}$, $\Tilde{\mathbb{C}}$ and $\Tilde{\mathcal{I}}$ denote the canonical extension of $\mathbb{N}$, $\mathbb{C}$, and $\mathcal{I}$, respectively, into $\Tilde{\mathcal{R}}$. We will also identify $F = \mathcal{R}^2$ and $\tilde{F} = \tilde{\mathcal{R}}^2$ and consider them as $\Gamma$-fields by setting $F_\Gamma$ to be the graph of $\mathrm{exp}_G$ restricted to $\mathcal{I}$ and $\tilde{F}_\Gamma$ as the graph of $\mathrm{exp}_G$ restricted to equivalence classes in $\tilde{F}$ corresponding to sequences with elements in $\tilde{\mathcal{I}}$.

A key note here is that there are only countably many definable functions in $\tilde{\mathcal{R}}$, giving the following. 

\begin{lemma}\label{dclCardinalityLemma}
    Let $C \subseteq \tilde{F}$, then $|\mathrm{dcl}(C \cup \mathbb{C})| \leq \mathrm{max}\{|C|, |\mathbb{C}|\}$.  
\end{lemma}
We now reintroduce, in our setting, the analogue of the holomorphic closure (see \cite{wilkie}) or exponential closure (see \cite{KirbyAlgebraic}) due to Utreras Alarcon in \cite[Chapter 5]{alarcon2015model}.  Henceforth, definable refers to being definable with respect to $\tilde{\mathcal{R}}$. Additionally, in this non-standard setting, and throughout this section, a holomorphic function refers to that of Peterzil and Starchenko defined in \cite[Section 2]{starchenko}.

\begin{definition}
    Let $A \subseteq \tilde{F}$. Define the \emph{definable holomorphic closure $\mathrm{hcl}^{\tilde{F}}(A)$ of $G$} as the set of all $x \in \tilde{F}$ such that there exists $n \in \mathbb{N}_0$, a $\mathbb{C}$-definable function $f:\tilde{F}^n \rightarrow\tilde{F}$ and $a_1, \ldots, a_n \in A$ where $f$ is holomorphic in a neighbourhood of $(a_1, \ldots, a_n)$ and $f(a_1, \ldots, a_n) = x$. 
\end{definition}

The above definition was only seen to be a finitary closure operator and, as such, Utreras Alarcon supplemented the above definition to define a pregeometry of the form below. 

\begin{definition}
    Given a finite $A \subseteq \tilde{F}$, the \emph{completed definable holomorphic closure $\overline{\mathrm{hcl}}^{\tilde{F}}(A)$ of $G$} is defined as the set of all $x \in\tilde{F}$ such that there exist an $n \in \mathbb{N}$ and a definable open set $D \subseteq\tilde{F}^n$ satisfying \begin{align*}
        x \in \bigcap\limits_{(a_1, \ldots, a_n) \in D}\mathrm{hcl}^{\tilde{F}}(A \cup \{a_1, \ldots, a_n\}).
    \end{align*} For general $A \subseteq\tilde{F}$, we define $\overline{\mathrm{hcl}}^{\tilde{F}}(A)$ as $\bigcup_{A_0 \subseteq A  \text{ finite}}\overline{\mathrm{hcl}}^{\tilde{F}}(A_0)$. 
\end{definition}

We now define derivations that will play a crucial role in giving properties of $\Gamma\mathrm{cl}$.
\begin{definition}
\label{derivationDefinition}
    A \emph{derivation on a subfield $K$ of $\tilde{F}$} is a map $\delta: K \rightarrow K$ such that, for all $a, b \in K$, \begin{itemize}
        \item $\delta(a + b) = \delta(a) + \delta(b)$
        \item $\delta(ab) = b\delta(a) + a\delta(b)$.
    \end{itemize} We call $\mathrm{ker}(\delta)$ the constant field of $\delta$. Let $f(z_1, \ldots, z_n)$ be a holomorphic $\tilde{F}$-valued function, $a = (a_1, \ldots, a_n) \in \tilde{F}^n$ and $\delta$ a derivation on $\tilde{F}$. We say that \emph{$\delta$ respects $f$ at $a$} if \begin{align*}
        \delta(f(a)) = \sum\limits_{i = 1}^n \frac{\partial f}{\partial z_i}(a) \cdot \delta(a_i) 
    \end{align*} and say that \emph{$\delta$ respects $f$} if \emph{$\delta$ respects $f$} for all $a \in \mathrm{dom}(f)$. 
\end{definition}



We now introduce a pregeometry due to Wilkie in \cite{wilkie}. 

\begin{definition}
    Let $\Delta$ denote the collection of all holomorphic $\tilde{F}$-valued functions definable over $\mathbb{C}$ defined in a neighbourhood of the origin and denote the set of all derivations that respect every function in $\Delta$ as $\mathrm{Der}_{\tilde{F}}(\Delta)$. For any $A \subseteq \tilde{F}$, we define $DD_\Delta(A)$ as the set of all $a \in \tilde{F}$ such that there exists a finite subset $A_0 \subseteq A$ such that every $\delta \in \mathrm{Der}_{\tilde{F}}(\Delta)$ that vanishes at $A_0$ will vanish at $a$.
\end{definition}

%

Utreras Alarcon in \cite[Theorem 4.4.3, p.~43]{alarcon2015model} showed that these pregeometries are indeed the same.

\begin{theorem}
\label{pregeometryDDALGAN}
$\overline{\mathrm{hcl}}^{\tilde{F}}$ and $DD_\Delta$ are pregeometries and $DD_\delta(A) = \overline{\mathrm{hcl}}^{\tilde{F}}(A)$ for all $A \subseteq \tilde{F}$. 
\end{theorem}


%

In \cite{KirbyAlgebraic}, Kirby used work of Ax in \cite{Ax} to prove a functional transcendence result and also did further work in \cite{semiAbelianKirby} to generalise the work of Ax to semi-abelian varieties. We, too, will make use of Kirby's generalisations. We first need the following definition. 

\begin{definition}
    Let $\mathcal{D}$ be a set of derivations on a field $K$ and let $\mathcal{D}_0 = \{\delta_1, \ldots, \delta_r\} \subseteq \mathcal{D}$ for some $r \in \mathbb{N}$. We define, for any $a = (a_1, \ldots, a_n) \in K^n$, \begin{align*}
        \mathrm{Jac}_{\mathcal{D}_0}(a) := \begin{pmatrix}
                                        \delta_1a_1 & \ldots & \delta_1 a_n\\
                                        \vdots & \ddots &  \vdots \\
                                        \delta_ra_1 & \ldots & \delta_ra_n
                                        \end{pmatrix}
    \end{align*} and define $\mathrm{rkJac}_{\mathcal{D}}(a) := \mathrm{max}\{\mathrm{rk} (\mathrm{Jac}_{\mathcal{D}'}(a)) \ | \ \mathcal{D}' \subseteq \mathcal{D} \text{ is finite}\}$.
\end{definition}

\begin{fact}[Ax-Schanuel \mbox{\cite[Theorem 3.8, p.~37]{semiAbelianKirby}}]\label{SchanuelGeneralization}
    Let $K$ be a field of characteristic $0$, let $\mathcal{D}$ be a collection of derivations on $K$, and let $C$ be the intersection of their constant fields. Let $\Omega_{G, N} \subseteq  T_0G^N \times G^N$ be the solution set given by the exponential differential equation of $G^N$ (that is, the intersection of the solution sets for each $\delta \in \mathcal{D}$). 
    Suppose that $(x, y) \in \Omega_{G, N}$ and $\mathrm{tr.deg}(x, y/C) - \mathrm{rkJac}_{\mathcal{D}}(x, y) < n$. Then there is a proper algebraic subgroup $H$ of $G^N$ and a constant point $\gamma \in TG^N(C)$ such that $(x, y)$ lies in the coset $\gamma \cdot TH$ where $TH$, $TG^N$ denotes the tangent bundle of $H$ and $G^N$, respectively.
\end{fact}

Inspired by arguments used in the proof of Proposition 10.7 in \cite[p.~44]{Bays_2018}, we can now relate $\Gamma\mathrm{cl}$ to $\mathrm{DD}_{\Delta}$. 
\begin{proposition}\label{gammaClosureRelatedToDerivations}
    Let $C \subseteq \tilde{F}$. Then $\Gamma\mathrm{cl}^{\tilde{F}}(C) \subseteq \mathrm{DD}_\Delta(C)$. 
\end{proposition}
\begin{proof}
    Suppose that $C$ is $\mathrm{DD}_\Delta$-closed. We consider $C$ as a $\Gamma$-field by taking $C_\Gamma$ to be ${\tilde{F}}_\Gamma \cap TG(C)$. Let $B$ be a proper exponentially finitely generated $\Gamma$-field extension of $C$ in $\tilde{F}$. Let $b = (b_1, \ldots, b_n) \in B_\Gamma^n$ be a basis for the extension over $C$ for some $n \in \mathbb{N}$ and let $V = \mathrm{Loc}(b/C)$ where $\mathrm{Loc}(b/C)$ denotes the smallest variety containing $b$ defined over $C$. Let $b_i = (b_{i, 1}, \mathrm{exp}(b_{i, 1}))$ for some $b_{i, 1} \in A$ for all $i \in \{1, \ldots, n\}$. We need only prove that $\delta(b/C) > 0$ which amounts to showing that $\mathrm{dim}V > n$. The $\mathrm{End}_0(G)$-linear independence of the $b_i$ implies that $b$ does not lie in any coset $\gamma \cdot TH$ where $H$ is a proper algebraic subgroup of $G^n$ and $\gamma$ is a point of $TH(C)$. Since $b$ satisfies the exponential differential equation of $G^n$ (see \cite[p.~3]{semiAbelianKirby}), applying Fact \ref{SchanuelGeneralization} gives \begin{align*}
    \mathrm{tr.deg}(b/C) - \mathrm{rkJac}_{\mathrm{Der}_{\tilde{F}}(\Delta)}(b) \geq n.
\end{align*} Since $B$ is a proper extension of $C$ and $C$ is $\mathrm{DD}_\Delta$-closed, it follows that there is a $\delta \in \mathrm{Der}_{\tilde{F}}(\Delta)$ such that $\delta(b_{i, 1}) \neq 0$ for some $i \in \{1, \ldots, n\}$. Thus, $\mathrm{rkJac}_{\mathrm{Der}_{\tilde{F}}(\Delta)}(b) \geq 1$ so that \begin{align*}
    \mathrm{tr.deg}(b/C) \geq n + 1 > n.
\end{align*} We deduce that $\mathrm{dim}V > n$.
\end{proof}

\begin{lemma}\label{algebraicAnalyticClosureRelatedToDefinableClosure}
    For any $C \subseteq \tilde{F}$, $\overline{\mathrm{hcl}}^{\tilde{F}}(C) \subseteq (\mathrm{dcl}(C \cup \mathbb{C}))^2$. 
\end{lemma}
\begin{proof}
    Let $C \subseteq \tilde{F}$. By definition of $\mathrm{hcl}$, it is clear that $\mathrm{hcl}(C) \subseteq (\mathrm{dcl}(C \cup \mathbb{C}))^{2}$. To show that $\overline{\mathrm{hcl}}^{\tilde{F}}(C) \subseteq (\mathrm{dcl}(C \cup \mathbb{C}))^2$, let $(a_1, a_2) \in \overline{\mathrm{hcl}}^{\tilde{F}}(C)$. By finite character of $\overline{\mathrm{hcl}}^{\tilde{F}}$, there exists $C_0 \subseteq C$ finite such that $(a_1, a_2) \in \overline{\mathrm{hcl}}^{\tilde{F}}(C_0)$. We want to show that $a_1, a_2 \in \mathrm{dcl}(C_0 \cup \mathbb{C})$. By definition of $\overline{\mathrm{hcl}}^{\tilde{F}}$, there exists a definable open set $D$ and $(x_1, \ldots, x_n) \in D$, with $x_i = (x_{i, 1}, x_{i, 2})$ for some $x_{i, 1}, x_{i, 2} \in \tilde{F}$ for all $i \in \{1, \ldots, n\}$, such that $(a_1, a_2) \in \mathrm{hcl}(C_0 \cup \{x_1, \ldots, x_n\})$. Hence, $(a_1, a_2) \in (\mathrm{dcl}(C_0 \cup \mathbb{C} \cup \{x_1, \ldots, x_n\}))^2$. We will only show that $a_1 \in \mathrm{dcl}(C_0 \cup \mathbb{C})$ as the case for $a_2$ is analogous. Assume that $a_1 \not\in \mathrm{dcl}(C_0 \cup \mathbb{C})$. We may also assume that $\mathrm{dim_\mathrm{dcl}}(x_{1, 1}, \ldots, x_{n, 2}/C_0 \cup \mathbb{C} \cup \{a_1\}) = 2n$. By exchange, we have $x_{i, j} \in \mathrm{dcl}(\{x_{1, 1}, \ldots, \hat{x_{i, j}}, \ldots, x_{n, 2}, a_1\} \cup C_0 \cup \mathbb{C})$ for some $i \in \{1, \ldots, n\}$ and $j \in \{1, 2\}$. This contradicts the $\mathrm{dcl}$-dimension of $\{x_1, \ldots, x_n\}$ over $\{a_1\} \cup C_0 \cup \mathbb{C}$. We conclude that $a_1 \in \mathrm{dcl}(C \cup \mathbb{C})$. 
\end{proof}

Putting together Theorem \ref{pregeometryDDALGAN}, Proposition \ref{gammaClosureRelatedToDerivations} and Lemma \ref{algebraicAnalyticClosureRelatedToDefinableClosure}, we have the following. 

\begin{corollary}\label{godTierCorollary}
    For any $C \subseteq \tilde{F}$, $\Gamma\mathrm{cl}^F(C) \subseteq (\mathrm{dcl}(C \cup \mathbb{C}))^2$.
\end{corollary}

The following result plays a key role in generalising the arguments of Boxall in \cite[Section 4]{boxall2024points}.

\begin{lemma}\label{keyLemma}
    Let $C \subseteq \Tilde{\mathbb{C}}$ be an algebraically closed subfield. Then $\Gamma\mathrm{cl}^{\Tilde{F}}(C) = C$. 
\end{lemma}
\begin{proof}
    We first show that $\langle C \cup \Gamma\mathrm{cl}^{\Tilde{F}}(C)_\Gamma, \Gamma\mathrm{cl}^{\Tilde{F}}(C)_\Gamma \rangle_{\Tilde{F}} = \Gamma\mathrm{cl}^{\Tilde{F}}(C)$ as $\Gamma$-fields. Since $C \subseteq \Tilde{\mathbb{C}}$, we may consider it as a $\Gamma$-subfield of $\Tilde{F}$ by setting $C_\Gamma = \{(0, 0_E)\}$. Now, using $K$ to denote $ \langle C \cup \Gamma\mathrm{cl}^{\Tilde{F}}(C)_\Gamma, \Gamma\mathrm{cl}^{\Tilde{F}}(C)_\Gamma \rangle_{\Tilde{F}}$, it follows that $K$ is a $\Gamma$-subfield of $\Tilde{F}$. By Definition \ref{gammaClosureDefinition}, we need only show that $K$ is $\Gamma$-closed. Let $B$ be an exponentially finitely generated $\Gamma$-field over $K$, with basis $b \in B_\Gamma^n$ for some $n \in \mathbb{N}$, such that $\delta(B/K) \leq 0$. Since $\Gamma\mathrm{cl}^{\tilde{F}}(C)_\Gamma = K_\Gamma$, it follows that \begin{align*}
        \mathrm{tr.deg}(b/K) - \mathrm{ldim}_{\mathrm{End}_0(G)}(b, \Gamma\mathrm{cl}^{\Tilde{F}}(C)_\Gamma) &= \mathrm{tr.deg}(b/K) - \mathrm{ldim}_{\mathrm{End}_0(G)}(b, K_\Gamma) = \delta(b/K) \leq 0
    \end{align*} which implies that \begin{align*}
        \mathrm{ldim}_{\mathrm{End}_0(G)}(b, \Gamma\mathrm{cl}^{\Tilde{F}}(C)_\Gamma) \geq \mathrm{tr.deg}(b/K) \geq \mathrm{tr.deg}(b/\Gamma\mathrm{cl}^{\Tilde{F}}(C)).
    \end{align*} Since $\Gamma\mathrm{cl}^{\Tilde{F}}(C)$ is $\Gamma$-closed, this implies that $b \in \Gamma\mathrm{cl}^{\Tilde{F}}(C)_\Gamma^N$ so that $K$ is $\Gamma$-closed. We deduce that $K = \Gamma\mathrm{cl}^{\Tilde{F}}(C)$.
    Combining $K = \Gamma\mathrm{cl}^{\Tilde{F}}(C)$ with, by Corollary \ref{godTierCorollary}, $\Gamma\mathrm{cl}^{\Tilde{F}}(C) \subseteq (\mathrm{dcl}(C \cup \mathbb{C}))^2$ and that $\Tilde{\mathcal{I}} \cap \Tilde{\mathbb{C}} = \{0\}$, we must have that $C = \Gamma\mathrm{cl}^{\Tilde{F}}(C)$.
\end{proof}

\section{Main result}\label{finitenessResultChapter}

To conclude proving Theorem \ref{OGEllipticConjecture}, it suffices, by Proposition \ref{thisIsIt}, to prove the following. 
We follow the same argument of \cite[Theorem 4.4, p.~12]{boxall2024points} with some minor changes reflected in our setting. We also further assume $\mathcal{R}$ to be a strongly $\kappa$-homogeneous and $\kappa$-saturated extension for some cardinal $\kappa > 2^{\aleph_0}$.  
\begin{theorem}\label{mainResult}
    Let $C_1, C_2 \subseteq G^3(\mathbb{C})$ be irreducible closed algebraic curves such that at least one of the curves is not defined over $\mathbb{Q}^\mathrm{alg}$. Suppose that there does not exist a proper algebraic subgroup $H \subseteq G^3(\mathbb{C})$ such that $C_1 \subseteq H$. Let $\mathcal{N} = \{n \in \mathbb{N} \ | \ [n]C_1 \subseteq C_2\}$. Then $\bigcup_{n \in \mathbb{N}\setminus\mathcal{N}}\{x \in C_1 \ | \ nx \in C_2\}$ is finite.
\end{theorem}
\begin{proof}
    Assume that $\bigcup_{n \in \mathbb{N}\setminus\mathcal{N}}\{x \in C_1 \ | \ nx \in C_2\}$ is infinite. We will argue for a contradiction. Let $e_1, \ldots, e_k \in \mathbb{C}$ be algebraically independent such that $C_1$ and $C_2$ are defined over $K = \mathrm{acl}(e_1, \ldots, e_k)$. Denoting $e_i = (e_{i, 1}, e_{i, 2})$ where $e_{i, 1}, e_{i, 2} \in \mathcal{R}$ for all $i \in \{1, \ldots, k\}$, we may assume that $e_{1,1}, e_{1,2} \ldots, e_{k,1}, e_{k,2}$ are $\mathrm{dcl}$-independent in the sense of the underlying real-closed field. We will break the argument up into various steps.

    \underline{Step 1}: We construct a field automorphism $\phi: F \rightarrow F$ that fixes $\mathbb{Q}^\mathrm{alg}$ where $\phi(e_1), \ldots, \phi(e_k)$ are $\mathrm{dcl}$-independent over $\mathbb{C}$, as a $2k$-tuple in the sense of the full o-minimal structure, and $\phi(e_i) - e_i \in \mathcal{I}$ for all $i \in \{1, \ldots, k\}$. 
    
    First, let $p_{1, 1}(x)$ be the partial type that expresses $x$ is $\mathrm{dcl}$-independent over $\mathbb{C}$, that $|x - e_{1, 1}| < \frac{1}{n}$ for all $n \in \mathbb{N}$ and that $x$ satisfies all formulae  $\varphi(x)$ with parameters in $\mathbb{Q}^\mathrm{rc}$ in the underlying real-closed field such that $\mathcal{R} \models \varphi(e_{1, 1})$. As $|\mathrm{dcl}(\mathbb{C})| \leq 2^{\aleph_0}$ by Lemma \ref{dclCardinalityLemma} and $\mathcal{R}$ is $\kappa$-saturated, we deduce that there exists $c_{1, 1} \in \mathcal{R}$ which satisfies the partial type $p_{1, 1}(x)$. We then construct the partial type $p_{1, 2}(x)$ which expresses that $x$ is $\mathrm{dcl}$-independent over $\mathbb{C} \cup \{c_{1, 1}\}$, that $|x - e_{1, 2}| < \frac{1}{n}$ for all $n \in \mathbb{N}$ and that $x$ satisfies all formulae $\varphi(c_{1,1}, x)$ with parameters in $\mathbb{Q}^\mathrm{rc}$ in the underlying real-closed field such that $\mathcal{R} \models \varphi(e_{1,1}, e_{1, 2})$. Similarly to above, as $|\mathrm{dcl}(\mathbb{C} \cup \{c_{1, 1}\})| \leq 2^{\aleph_0}$, again, by Lemma \ref{dclCardinalityLemma} and $\mathcal{R}$ is $\kappa$-saturated there exists $c_{1, 2} \in \mathcal{R}$ which satisfies the partial type $p_{1, 2}(x)$. We proceed in this manner, by induction, so that there exists $c_{1, 1}, c_{1, 2}, \ldots, c_{k, 2} \in \mathcal{R}$ such that $\mathcal{R} \models p_{i, j}(c_{i, j})$ for all $i \in \{1, \ldots, k\}$ and $j \in \{1,2\}$. Furthermore, by our construction of $p_{i, j}$, $c_{i, j}$ has the same type as $e_{i, j}$, in the sense of the underlying real-closed field, for all $i \in \{1, \ldots, k\}$ and $j \in \{1,2\}$. As $\mathcal{R}$ is strongly $\kappa$-homogeneous, it follows that there exists an isomorphism $\psi: \mathcal{R} \rightarrow \mathcal{R}$, in the sense of real-closed fields, such that $\psi(e_{i, j}) = c_{i, j}$ for all $i \in \{1, \ldots, k\}$ and $j \in \{1,2\}$ and such that $\psi$ fixes $\mathbb{Q}^\mathrm{rc}$. Defining $\phi: F \rightarrow F$ by sending $a = (a_1, a_2) \in F$ to $(\psi(a_1), \psi(a_2))$ gives our wanted isomorphism.     
    
    Step 1 allows us to construct infinitesimal shifts of $C_1$ and $C_2$. Let $K' = \phi(K)$, $C_1' = \phi(C_1(F))$, $C_2' = \phi(C_2(F))$ and $e_i' = \phi(e_i)$ for all $i \in \{1, \ldots, k\}$. As $C_1, C_2 \subseteq G^3(F)$, $\phi$ fixes $\mathbb{Q}^\mathrm{alg}$ and $G$ is assumed to be defined over $\mathbb{Q}^\mathrm{alg}$, we have that $C_1', C_2' \subseteq G^3(F)$. Note that, since $e_1', \ldots, e_k'$ are $\mathrm{dcl}$-independent over $\mathbb{C}$, it follows that they are $\mathrm{dcl}$-independent over $\Tilde{\mathbb{C}}$. 
    
    \underline{Step 2}: Given any $n \in \mathbb{N}\setminus\mathcal{N}$, $x = (x_1, x_2, x_3) \in C_1$ and $a = (a_1, a_2, a_3) \in \mathbb{C}^3$ such that $nx \in C_2$ and $\mathrm{exp}_G(a) = x$, we want to construct some $z = (z_1, z_2, z_3) \in C_1'(F)$ and $b = (b_1, b_2, b_3) \in F^3$ such that $nz \in C_2'(F)$, $a_i - b_i \in \mathcal{I}$ and $\mathrm{exp}_G(b_i) = z_i$ for all $i \in \{1, 2, 3\}$. 
    
    As a convention, we set $\mathrm{exp}_G(0) = 0_G$. 
    If $x_i = 0_G$, then we set $a_i = 0$, $b_i = 0$ and $z_i = 0_G$ for all $i \in \{1, 2, 3\}$. We form the partial type which has the formulas that say that, for each real coordinate of non-zero $a_1, a_2, a_3$, you specify all rational numbers that the real coordinate lies between and that $\mathrm{exp}_G(a) \in C_1$ and $\mathrm{exp}_G(na) \in C_2$. Since $e_1, \ldots, e_k$ has the same type as $e_1', \ldots, e_k'$, using the partial type above, there exists $z = (z_1, z_2, z_3) \in C_1'$ and $b = (b_1, b_2, b_3) \in F^3$ such that $nz \in C_2'$, $a_i - b_i \in \mathcal{I}$ and $\mathrm{exp}_G(b_i) = z_i$ for all $i \in \{1, 2, 3\}$. 

    \underline{Step 3}: Using step 2, we construct elements in $C_1(\Tilde{\mathbb{C}})$ and in $C_1'(\Tilde{F})$ that we use to obtain a contradiction.
    
     Since $\bigcup_{n \in \mathbb{N}\setminus\mathcal{N}}\{x \in C_1 \ | \ nx \in C_2\}$ is infinite, by Lemma \ref{infiniteIndexes}, there exists an increasing sequence of natural numbers $(n_i)_{i \in \mathbb{N}}$ such that \begin{itemize}
        \item there exists $x_i = (x_{i, 1}, x_{i, 2}, x_{i,3}) \in C_1(\mathbb{C})$ with $n_ix_i \in C_2(\mathbb{C})$ and
        \item there exists $a_i = (a_{i, 1}, a_{i, 2}, a_{i, 3}) \in \mathbb{C}^3$ with $\mathrm{exp}_G(a_i) = x_i$ and $\mathrm{exp}_G(n_ia_i) = n_ix_i$.
    \end{itemize} Furthermore, by step 2, \begin{itemize}
        \item there exists $z_i = (z_{i, 1}, z_{i, 2}, z_{i,3}) \in C_1'(F)$ with $n_iz_i \in C_2'(F)$ and
        \item there exists $b_i = (b_{i, 1}, b_{i, 2}, b_{i, 3}) \in F^3$ with $\mathrm{exp}_G(b_i) = z_i$ and $\mathrm{exp}_G(n_ib_i) = n_iz_i$ such that $a_i - b_i \in \mathcal{I}^3$.
    \end{itemize} We can now define, in $\Tilde{\mathcal{R}}$, \begin{itemize}
        \item $\Tilde{n} := [(n_i)_{i \in \mathbb{N}}]$,
        \item $\Tilde{x} := [(x_i)_{i \in \mathbb{N}}]$,
        \item $\Tilde{a} := [(a_i)_{i \in \mathbb{N}}]$,
        \item $\Tilde{z} := [(z_i)_{i \in \mathbb{N}}]$ and
        \item $\Tilde{b} := [(a_i)_{i \in \mathbb{N}}]$.
    \end{itemize} It follows, by the properties of ultrapowers, that $\Tilde{x} \in C_1(\Tilde{\mathbb{C}})$ and $\Tilde{z} \in C_1'(\Tilde{F})$ such that $\Tilde{n}\Tilde{x} \in C_2(\Tilde{\mathbb{C}})$, $\Tilde{n}\Tilde{z} \in C_2'(\Tilde{F})$, $\mathrm{exp}_G(\Tilde{a}) = \Tilde{x}$, $\mathrm{exp}_G(\Tilde{b}) = \Tilde{z}$ and $\Tilde{a} - \Tilde{b} \in \Tilde{\mathcal{I}}^3$.    

    \underline{Step 4}: Use properties of $\Tilde{z}$ and $\Tilde{x}$ to obtain a contradiction. Consider the point $\Tilde{z} - \Tilde{x}$ which lies on the curve $C_1'(\Tilde{F}) - \Tilde{x}$ and the point $\Tilde{n}\Tilde{z} - \Tilde{n}\Tilde{x}$ which lies on the curve $C_2'(\Tilde{F}) - \Tilde{n}\Tilde{x}$. It follows that both curves are defined over a finitely generated field $L$ over $K'$ given by generators $l_1, \ldots, l_t \in \tilde{\mathbb{C}}$ for some $t \in \mathbb{N}$ (see \cite[Proposition 2.6]{Silverman:1338326} for example). 
    We now give various cases to obtain a contradiction.

    \underline{Case 1}: Suppose that $\Tilde{z} - \Tilde{x}$ is additively independent. 
    Let $u$ be a maximal sub-tuple of the tuple $(l_1, \ldots, l_t)$ such that $u$ is $\mathrm{acl}$-independent over $K \cup \{\tilde{n}\}$. 
    As $C_1'(\Tilde{F}) - \Tilde{x}$ is a curve defined over $L$, it follows that \begin{align*}
        \mathrm{tr.deg}(\Tilde{z} - \Tilde{x}/L) \leq 1.
    \end{align*} Similarly, we have \begin{align*}
        \mathrm{tr.deg}(\Tilde{n}\Tilde{z} - \Tilde{n}\Tilde{x}/L) \leq 1.
    \end{align*} Since $L \subseteq \mathrm{acl}(\Gamma\mathrm{cl}(K', K, u), \Tilde{n})$, it follows that \begin{align*}
        \mathrm{tr.deg}(\Tilde{z} - \Tilde{x}, \Tilde{n}\Tilde{z} - \Tilde{n}\Tilde{x}/\Gamma\mathrm{cl}(K', K, u), \Tilde{n}) \leq 2.
    \end{align*} By the manner in which we defined $\Tilde{z}$ and $\Tilde{n}\Tilde{z}$ in step 3, we have that $\Tilde{z} - \Tilde{x} = \mathrm{exp}_G(\Tilde{b} - \Tilde{a})$ and $\Tilde{n}\Tilde{z} - \Tilde{n}\Tilde{x} = \mathrm{exp}_G(\Tilde{n}\Tilde{b} - \Tilde{n}\Tilde{a}) = \mathrm{exp}_G(\Tilde{n}(\Tilde{b} - \Tilde{a}))$ where $\Tilde{b} - \Tilde{a}, \ \Tilde{n}(\Tilde{b} - \Tilde{a}) \in \Tilde{\mathcal{I}}^3$. Thus, by Theorem \ref{additiveIndependenceTheorem}, we must have $\Tilde{n} \in \Gamma\mathrm{cl}(K', K, u)$. Since $e_1', \ldots, e_k'$ are $\mathrm{dcl}$-independent over $\Tilde{\mathbb{C}}$, they are $\Gamma\mathrm{cl}$-independent over $\Tilde{\mathbb{C}}$ so that, by exchange, we must have $\Tilde{n} \in \Gamma\mathrm{cl}(K, u)$. By Lemma \ref{keyLemma}, $\Tilde{n} \in \mathrm{acl}(K, u)$. Since $\Tilde{n}$ is infinite, it is immediate that $\Tilde{n} \not\in \mathrm{acl}(K)$. Using the exchange principle once more, we contradict the $\mathrm{acl}$-independence of $u$ over $K \cup \{\tilde{n}\}$.

    \underline{Case 2}: Suppose that $\Tilde{z} - \Tilde{x}$ has additive dimension $2g$. Then there is an algebraic subgroup $H \subseteq G^3(\Tilde{F})$ of dimension $2g$ such that $\Tilde{z} - \Tilde{x} \in H$. By Fact \ref{subgroupSatisfiesAdditiveRelations}, 
    we may assume $H$ is defined using elements of $\mathrm{End}(G)$ which implies that $z_i - x_i \in H$ for all $i \in I$ for some $I \subseteq \mathcal{U}$. It follows that $n_iz_i - n_ix_i \in H$ for all $i \in I$, so that $\Tilde{n}\Tilde{z} - \Tilde{n}\Tilde{x} \in H$. We deduce that $\Tilde{x} \in H + \Tilde{z}$ and $\Tilde{n}\Tilde{x} \in H + \Tilde{n}\Tilde{z}$.

    \underline{Subcase 2a}: Suppose that both $C_1(\Tilde{F}) \subseteq H + \Tilde{z}$ and $C_2(\Tilde{F}) \subseteq H + \Tilde{n}\Tilde{z}$. This implies, where $\lambda_{1}, \lambda_{2}, \lambda_{3} \in \mathrm{End}(G)$, that \begin{align*}
        \sum\limits_{j = 1}^{3}\lambda_{j}(\Tilde{x}_j) = \Tilde{y}
    \end{align*} for some $\Tilde{y} \in G(\Tilde{F})$. Since $\bigcup_{m \in \mathbb{N}\setminus \mathcal{N}}\{x \in C_1 \ | \ mx \in C_2\}$ is infinite, $\Tilde{y}$ must be a torsion point. This contradicts the fact that $C_1$ is not contained in a proper algebraic subgroup of $G^3(\mathbb{C})$.

    \underline{Subcase 2b}: Suppose that $C_2(\Tilde{F}) \nsubseteq H + \Tilde{n}\Tilde{z}$. We must have that $C_2(\Tilde{F}) \cap H + \Tilde{n}\Tilde{z}$ at only finitely many points. If this were not the case, $C_2(\Tilde{F}) \cap H + \Tilde{n}\Tilde{z}$ would have irreducible component $C_3(\Tilde{F})$ of the same dimension. If $C_2(\Tilde{F}) \neq C_3(\Tilde{F})$, then this would contradict the dimension of $C_2(\Tilde{F})$. So $C_2(\Tilde{F}) = C_3(\Tilde{F})$ which contradicts that $C_2(\Tilde{F}) \nsubseteq H + \Tilde{n}\Tilde{z}$. 
    We deduce that $\Tilde{n} \Tilde{x} \in \mathrm{acl}(\Tilde{n}\Tilde{z}, K, K')$. We now claim, using exchange, that $\Tilde{n}\Tilde{z} \in \mathrm{acl}(\Tilde{n}\Tilde{x}, K, K')$. By construction, $\Tilde{n}\Tilde{x} \not\in \mathrm{acl}(K)$ unless $\Tilde{n}\Tilde{x}$ is a torsion point. If $\Tilde{n}\Tilde{x}$ was torsion, then there would be infinitely many $x_i \in C_1(F)$ that are torsion points. By Corollary \ref{torsionCorollary}, this would imply that $C_1$ is contained in a proper algebraic subgroup of $G^3(\mathbb{C})$ which is a contradiction. Now, as $e_1', \ldots, e_k'$ are $\mathrm{dcl}$-independent over $\Tilde{\mathbb{C}}$, and hence $\mathrm{acl}$-independent over $\Tilde{\mathbb{C}}$, it follows that $\Tilde{n}\Tilde{x} \not\in \mathrm{acl}(K, K')$. Hence, the claim follows. 
We deduce that \begin{align*}
        \mathrm{tr.deg}(\Tilde{n}\Tilde{z} - \Tilde{n}\Tilde{x}/K, K', \Tilde{n}\Tilde{x}) = 0.
    \end{align*} 
    As in Case 1, let $u$ be a $\mathrm{acl}$-basis for $(l_1, \ldots, l_t)$ over $K \cup \{\Tilde{n}\}$. Combining Theorem \ref{additiveIndependenceTheorem} with the bound \begin{align*}
        \mathrm{tr.deg}(\Tilde{z} - \Tilde{x}, \Tilde{n}\Tilde{z} - \Tilde{n}\Tilde{x}/\Gamma\mathrm{cl}(K, K', u), \Tilde{n}) \leq 1,
    \end{align*} we have that $\Tilde{n} \in \Gamma\mathrm{cl}(K, K', u)$. As in Case 1, we obtain a contradiction.

    \underline{Subcase 2c}: Supposing that $C_1(\Tilde{F}) \nsubseteq H + \Tilde{z}$ leads to an analogous contradiction to that in Subcase 2b where we instead have $\Tilde{x} \not\in \mathrm{acl}(K, K')$.

    \underline{Case 3}: Suppose that $\Tilde{z} - \Tilde{x}$ has additive dimension $g$. Then there is an algebraic subgroup $H \subseteq G^3(\Tilde{F})$ of dimension $g$ such that $\Tilde{x} \in H + \Tilde{z}$ and $\Tilde{n}\Tilde{x} \in H + \Tilde{n}\Tilde{z}$. By Proposition \ref{propositionIForgot}, $C_1(\Tilde{F}) \nsubseteq H + \Tilde{z}$ and $C_2(\Tilde{F}) \nsubseteq G + \Tilde{n}\Tilde{z}$. So $\Tilde{x} \in \mathrm{acl}(\Tilde{z}, K$ and $\Tilde{n}\Tilde{x} \in \mathrm{acl}(\Tilde{n}\Tilde{z}, K)$. Using similar reasoning to Subcase 2b, we have that $\Tilde{z} \in \mathrm{acl}(\Tilde{x}, K, K')$ and $\Tilde{n}\Tilde{z} \in \mathrm{acl}(\Tilde{n}\Tilde{x}, K, K')$. Let $u$ be a $\mathrm{acl}$-basis for $(l_1, \ldots, l_t)$ over $K \cup \{\Tilde{n}\}$. Once again, combining Theorem \ref{additiveIndependenceTheorem} with the bound \begin{align*}
        \mathrm{tr.deg}(\tilde{z} - \tilde{x}, \tilde{n}\tilde{z} - \tilde{n}\tilde{x}/\Gamma\mathrm{cl}(K, K', u), \tilde{n}) = 0,
    \end{align*} we have that $\tilde{n} \in \Gamma\mathrm{cl}(K, K', u)$. We obtain a contradiction as in the previous cases.

    \underline{Case 4}: Suppose that $\Tilde{z} - \Tilde{x}$ has additive dimension $0$. Then $\Tilde{z} - \Tilde{x} = (0_G, 0_G, 0_G)$, so that $z_i = x_i$ for infinitely many $i \in \mathbb{N}$. Thus, the curves $C_1(F)$ and $C_1'(F)$ have infinitely many points in common. It follows, by similar arguments in Subcase 2b, 
    that $C_1(F) = C_1'(F)$. Similarly, $C_2(F) = C_2'(F)$. Then $C_1(F)$ and $C_2(F)$ are both defined over $K \cap K' = \mathbb{Q}^\mathrm{alg}$ which is a contradiction. 
\end{proof}

\printbibliography

\end{document}